\documentclass[reqno,11pt,a4paper]{amsart}

\usepackage{amssymb}
\usepackage[margin=2.7cm]{geometry}
\usepackage{tikz}
\usepackage{mathtools}
\usepackage{pgfplots}
\pgfplotsset{width=7cm,compat=1.9}
\usepackage{subfig}
\usetikzlibrary{shapes.misc}
\usepackage{hyperref}
\hypersetup{
    colorlinks,
    linkcolor={red!50!black},
    citecolor={blue!50!black},
    urlcolor={blue!80!black}
}

\graphicspath{{./Plots/}}
\tikzset{cross/.style={cross out, draw=black, minimum size=2*(#1-\pgflinewidth), inner sep=0pt, outer sep=0pt}, cross/.default={1pt}}

\newcommand{\ZZ}{\mathbb{Z}}
\newcommand{\RR}{\mathbb{R}}
\newcommand{\NN}{\mathbb{N}}
\newcommand{\PP}{\mathbb{P}}
\newcommand{\cT}{\mathcal{T}}
\newcommand{\cS}{\mathcal{S}}
\DeclareMathOperator{\GL}{GL}
\DeclareMathOperator{\conv}{conv}
\DeclareMathOperator{\width}{width}
\DeclareMathOperator{\Aut}{Aut}
\DeclareMathOperator{\Ehr}{Ehr}
\DeclareMathOperator{\ehr}{ehr}
\DeclareMathOperator{\Hom}{Hom}

\newtheorem{theorem}{Theorem}[section]
\newtheorem{proposition}[theorem]{Proposition}
\newtheorem{lemma}[theorem]{Lemma}
\newtheorem{corollary}[theorem]{Corollary}
\theoremstyle{definition}
\newtheorem{definition}[theorem]{Definition}

\begin{document}

\author[G.\ Hamm]{Girtrude~Hamm}
\address{School of Mathematical Sciences\\University of Nottingham\\Nottingham\\NG7 2RD\\UK}
\email{girtrude.hamm@nottingham.ac.uk}


\title{Classification of lattice triangles by their two smallest widths}

\begin{abstract}
We introduce the notion of the second lattice width of a lattice polytope and use this to classify lattice triangles by their width and second width.
This is equivalent to classifying lattice triangles contained in a given rectangle (and no smaller rectangle) up to affine equivalence.
Using this classification we investigate the automorphism groups and Ehrhart theory of lattice triangles.
We also show that the sequence counting lattice triangles contained in dilations of the unit square has generating function equal to the Hilbert series of a degree 8 hypersurface in \(\PP(1,1,1,2,2,2)\).
\end{abstract}
\maketitle

\section{Introduction}
\label{sec:intro}

A \emph{lattice point} is a point in a lattice \(\ZZ^d\).
A \emph{lattice triangle} \(T \subseteq \RR^2\) is the convex hull of three (non-colinear) lattice points, which we call the vertices of \(T\).
More generally, a \emph{lattice polytope} \(P \subseteq \RR^d\) is the convex hull of finitely many lattice points \(v_1, \dots, v_k\) in \(\ZZ^d\).
It is called a \emph{lattice polygon} if it is two-dimensional and it is called a \emph{lattice simplex} if the points \(v_i\) are affinely independent.
We consider lattice polytopes as being defined up to affine unimodular equivalence.
Two lattice polytopes \(P\) and \(P'\) are said to be \emph{affine equivalent} if \(P\) can be mapped to \(P'\) by a change of basis of \(\ZZ^d\) followed by an integral translation.

Recall that for a lattice polytope \(P \subseteq \RR^d\) and a primitive dual vector \(u \in (\ZZ^d)^*\) the width of \(P\) with respect to \(u\) is the lattice length of the interval obtained by projecting \(P\) along the hyperplane with normal vector \(u\); that is, \(\width_u(P) \coloneqq \max_{x \in P}\{u \cdot x\}-\min_{x \in P}\{u \cdot x\}\).
Then the \emph{(first) width} of \(P\), written \(\width^1(P)\), is the minimum width along all non-zero dual vectors \(u\), i.e. \(\width^1(P) \coloneqq \min_{u \in (\ZZ^d)^*\setminus \{0\}} \{\width_u(P)\}\).

Restricting the width of polytopes is a powerful tool towards classifying them.
It was shown in \cite{width_bound} that in each dimension \(d\) there is some constant \( w \in \NN\) such that the number of lattice polytopes with width larger than \(w\) containing \(n\) lattice points is finite.
When \(d >2\) there are infinitely many polytopes with size \(n\) and small width so to classify them one must classify the finite exceptional polytopes with large width and the infinitely many polytopes with small width.
This was done, for example, in the classification of empty four-dimensional lattice simplices \cite{empty_dim_4_simpl}.
Although there are finitely many polygons containing a given number of lattice points there are still polygon classifications which follow a similar pattern. 
For example, the classification of lattice polygons which contain a point from which all other lattice points are visible \cite{all_lattice_pts_visible} and the classification of rational polygons \(P\) such that \(rP\) is a lattice polygon for fixed \(r \in \NN\) both consist of finitely many exceptional polygons and infinitely many polygons with small width.
Here we present a method for classifying infinitely many lattice polytopes with a given width by applying it to the case of lattice triangles.

Up to equivalence, there are infinitely many lattice polytopes with a given width, in dimension at least 2.
In dimension 2, to have a finite collection to classify we introduce the second width.
For a lattice polygon \(P\), we can choose linearly independent directions \(u_1, u_2 \in (\ZZ^2)^*\) such that the tuple \((w_1,w_2)\) of widths with respect to \(u_1\) and \(u_2\) is as small as possible considered under lexicographic order.
We will refer to \(w_2\) as the \emph{second width} of \(P\) written \(\width^2(P)\).
The second width is never less than the first width, so from now on, unless otherwise stated, let \(w_1\) and \(w_2\) be positive integers satisfying \(w_1 \leq w_2\).
We prove in Proposition~\ref{prop:lex_eq_width} that a polygon is always equivalent to a subset of a \(w_1 \times w_2\) rectangle so there are finitely many lattice polygons up to equivalence with first and second width \(w_1\) and \(w_2\).

The main goal of this paper will be to describe the finite set
\[
\cT_{w_1,w_2} \coloneqq \{\conv(v_1,v_2,v_3) : v_1,v_2,v_3 \in \ZZ^2, \width^1(T)=w_1, \width^2(T)=w_2\}/ \sim
\]
of lattice triangles whose first and second widths are \(w_1\) and \(w_2\), where \(\sim\) denotes affine equivalence.
Theorem~\ref{thm:main} achieves this goal by establishing a bijection between \(\cT_{w_1,w_2}\) and the set \(\cS_{w_1,w_2}\) defined as follows.
\begin{definition}\label{def:main}
Let \(\cS_{w_1,w_2}\) be the set of lattice triangles \(T\) such that
\begin{itemize}
    \item If \(w_1<w_2\) then \(T\) is equal to
    \begin{itemize}
    \item[(A)] \(\conv((0,0),(w_1,y_1),(0,w_2))\) where \(0 \leq y_1 \leq (w_2-y_1 \mod w_1)\) or
    \item[(B)] \(\conv((0,0),(w_1,y_1),(x_2,w_2))\) where \(0 < x_2 \leq \frac{w_1}{2}\) and \(0 \leq y_1 \leq w_1-x_2\) or
    \item[(C)] \(\conv((0,y_0),(w_1,0),(x_2,w_2))\) where \(1<x_2 < \frac{w_1}{2}\) and \(0<y_0 <x_2\).
    \end{itemize}
    \item If \(w_1=w_2\) then \(T\) is equal to
    \begin{itemize}
    \item[(A)] \(\conv((0,0),(w_1,y_1),(0,w_1))\) where \(0\leq y_1 \leq (w_1-y_1 \mod w_1)\) or
    \item[(B)] \(\conv((0,0),(w_1,y_1),(x_2,w_1))\) where \(0<x_2\leq \frac{w_1}{2}\) and \(x_2 \leq y_1 \leq w_1-x_2\).
    \end{itemize}
\end{itemize}
\end{definition}
These triangles are each of one of the three types depicted in Figure~\ref{fig:types}.
The triangles in \(\cS_{w_1,w_2}\) for \(w_1 \leq w_2 \leq 4\) are depicted in Figure~\ref{fig:many_triangles}.
\begin{figure}[ht]
    \subfloat[]{
    \begin{tikzpicture}[x=0.5cm,y=0.5cm]
    \draw[dashed] (0,0) -- (3,0) -- (3,4) -- (0,4) -- cycle;
    \filldraw[fill=gray!40] (0,0) -- (0,4) -- (3,1) -- cycle;
    \draw[dotted] (1.5,0) -- (1.5,4);
    \node[draw,circle,inner sep=1.2pt,fill] at (0,0) {};
    \node[draw,circle,inner sep=1.2pt,fill] at (0,4) {};
    \node[draw,circle,inner sep=1.2pt,fill=white] at (3,1) {};
    \end{tikzpicture}
    }\qquad
    \subfloat[]{
    \begin{tikzpicture}[x=0.5cm,y=0.5cm]
    \draw[dashed] (0,0) -- (3,0) -- (3,4) -- (0,4) -- cycle;
    \filldraw[fill=gray!40] (0,0) -- (1,4) -- (3,1) -- cycle;
    \draw[dotted] (1.5,0) -- (1.5,4);
    \node[draw,circle,inner sep=1.2pt,fill] at (0,0) {};
    \node[draw,circle,inner sep=1.2pt,fill=white] at (1,4) {};
    \node[draw,circle,inner sep=1.2pt,fill=white] at (3,1) {};
    \end{tikzpicture}    
    }\qquad
    \subfloat[]{
    \begin{tikzpicture}[x=0.5cm,y=0.5cm]
    \draw[dashed] (0,0) -- (3,0) -- (3,4) -- (0,4) -- cycle;
    \filldraw[fill=gray!40] (0,1) -- (1,4) -- (3,0) -- cycle;
    \draw[dotted] (1.5,0) -- (1.5,4);
    \node[draw,circle,inner sep=1.2pt,fill=white] at (0,1) {};
    \node[draw,circle,inner sep=1.2pt,fill=white] at (1,4) {};
    \node[draw,circle,inner sep=1.2pt,fill] at (3,0) {};
    \end{tikzpicture}
     }
    \centering
    \caption{The three types of triangle in the set \(\cS_{w_1,w_2}\) as in Definition~\ref{def:main}.
    The rectangles are \([0,w_1] \times [0,w_2]\) and the vertical dotted lines divide them in half.
    Black vertices are fixed for a given type while white vertices vary within fixed ranges along the boundary of the rectangle.
    }
    \label{fig:types}
\end{figure}
\begin{theorem} \label{thm:main}
There is a bijection from \(\cS_{w_1,w_2}\) to \(\cT_{w_1,w_2}\) given by the map taking \(T\) to its affine equivalence class.
In particular,
\begin{itemize}
    \item when \(w_1<w_2\) the cardinality of \(\cT_{w_1,w_2}\) is
    \begin{itemize}
        \item \(\frac{w_1^2}{2} +2\) if \(w_1\) and \(w_2\) are even
        \item \(\frac{w_1^2}{2} +1\) if \(w_1\) is even and \(w_2\) is odd
        \item \(\frac{w_1^2}{2} + \frac12\) if \(w_1\) is odd
    \end{itemize}
    \item and when \(w_1=w_2\) the cardinality of \(\cT_{w_1,w_1}\) is
    \begin{itemize}
        \item \(\frac{w_1^2}{4} + \frac{w_1}{2} +1\) if \(w_1\) is even
        \item \(\frac{w_1^2}{4} + \frac{w_1}{2} + \frac14\) if \(w_1\) is odd
    \end{itemize}
\end{itemize}
\end{theorem}

\begin{figure}[ht]
    \centering
\begin{tikzpicture}[x=0.5cm,y=0.5cm]
\node[draw,circle,inner sep=1pt,fill] at (0,0) {};
\node[draw,circle,inner sep=2pt] at (0,0) {};
\node[draw,circle,inner sep=3pt] at (0,0) {};
\draw (0,0-2) -- (0,1-2);
\node[draw,circle,inner sep=1pt,fill] at (0,0-2) {};
\node[draw,circle,inner sep=2pt] at (0,0-2) {};
\node[draw,circle,inner sep=1pt,fill] at (0,1-2) {};
\draw (0,0-5) -- (0,2-5);
\node[draw,circle,inner sep=1pt,fill] at (0,0-5) {};
\node[draw,circle,inner sep=2pt] at (0,0-5) {};
\node[draw,circle,inner sep=1pt,fill] at (0,2-5) {};

\draw (0,0-8) -- (0,2-8);
\node[draw,circle,inner sep=1pt,fill] at (0,0-8) {};
\node[draw,circle,inner sep=1pt,fill] at (0,1-8) {};
\node[draw,circle,inner sep=1pt,fill] at (0,2-8) {};
\draw (0,0-12) -- (0,3-12);
\node[draw,circle,inner sep=1pt,fill] at (0,0-12) {};
\node[draw,circle,inner sep=2pt] at (0,0-12) {};
\node[draw,circle,inner sep=1pt,fill] at (0,3-12) {};

\draw (0,0-16) -- (0,3-16);
\node[draw,circle,inner sep=1pt,fill] at (0,0-16) {};
\node[draw,circle,inner sep=1pt,fill] at (0,1-16) {};
\node[draw,circle,inner sep=1pt,fill] at (0,3-16) {};
\draw (0,0-21) -- (0,4-21);
\node[draw,circle,inner sep=1pt,fill] at (0,0-21) {};
\node[draw,circle,inner sep=2pt] at (0,0-21) {};
\node[draw,circle,inner sep=1pt,fill] at (0,4-21) {};

\draw (0,0-26) -- (0,4-26);
\node[draw,circle,inner sep=1pt,fill] at (0,0-26) {};
\node[draw,circle,inner sep=1pt,fill] at (0,1-26) {};
\node[draw,circle,inner sep=1pt,fill] at (0,4-26) {};

\draw (0,0-31) -- (0,4-31);
\node[draw,circle,inner sep=1pt,fill] at (0,0-31) {};
\node[draw,circle,inner sep=1pt,fill] at (0,2-31) {};
\node[draw,circle,inner sep=1pt,fill] at (0,4-31) {};
\draw[fill=gray!30] (0+1,0-2) -- (1+1,0-2) -- (0+1,1-2) -- cycle;
\node[draw,circle,inner sep=1pt,fill] at (0+1,0-2) {};
\node[draw,circle,inner sep=1pt,fill] at (1+1,0-2) {};
\node[draw,circle,inner sep=1pt,fill] at (0+1,1-2) {};
\draw[fill=gray!30] (0+1,0-5) -- (1+1,0-5) -- (0+1,2-5) -- cycle;
\node[draw,circle,inner sep=1pt,fill] at (0+1,0-5) {};
\node[draw,circle,inner sep=1pt,fill] at (1+1,0-5) {};
\node[draw,circle,inner sep=1pt,fill] at (0+1,2-5) {};
\draw[fill=gray!30] (0+1,0-12) -- (1+1,0-12) -- (0+1,3-12) -- cycle;
\node[draw,circle,inner sep=1pt,fill] at (0+1,0-12) {};
\node[draw,circle,inner sep=1pt,fill] at (1+1,0-12) {};
\node[draw,circle,inner sep=1pt,fill] at (0+1,3-12) {};
\draw[fill=gray!30] (0+1,0-21) -- (1+1,0-21) -- (0+1,4-21) -- cycle;
\node[draw,circle,inner sep=1pt,fill] at (0+1,0-21) {};
\node[draw,circle,inner sep=1pt,fill] at (1+1,0-21) {};
\node[draw,circle,inner sep=1pt,fill] at (0+1,4-21) {};
\draw[fill=gray!30] (0+3,0-5) -- (2+3,0-5) -- (0+3,2-5) -- cycle;
\node[draw,circle,inner sep=1pt,fill] at (0+3,0-5) {};
\node[draw,circle,inner sep=1pt,fill] at (2+3,0-5) {};
\node[draw,circle,inner sep=1pt,fill] at (0+3,2-5) {};

\draw[fill=gray!30] (0+6,0-5) -- (2+6,1-5) -- (0+6,2-5) -- cycle;
\node[draw,circle,inner sep=1pt,fill] at (0+6,0-5) {};
\node[draw,circle,inner sep=1pt,fill] at (2+6,1-5) {};
\node[draw,circle,inner sep=1pt,fill] at (0+6,2-5) {};

\draw[fill=gray!30] (0+3,0-8) -- (2+3,1-8) -- (1+3,2-8) -- cycle;
\node[draw,circle,inner sep=1pt,fill] at (0+3,0-8) {};
\node[draw,circle,inner sep=1pt,fill] at (2+3,1-8) {};
\node[draw,circle,inner sep=1pt,fill] at (1+3,2-8) {};
\draw[fill=gray!30] (0+3,0-12) -- (2+3,0-12) -- (0+3,3-12) -- cycle;
\node[draw,circle,inner sep=1pt,fill] at (0+3,0-12) {};
\node[draw,circle,inner sep=1pt,fill] at (2+3,0-12) {};
\node[draw,circle,inner sep=1pt,fill] at (0+3,3-12) {};

\draw[fill=gray!30] (0+6,0-12) -- (2+6,0-12) -- (1+6,3-12) -- cycle;
\node[draw,circle,inner sep=1pt,fill] at (0+6,0-12) {};
\node[draw,circle,inner sep=1pt,fill] at (2+6,0-12) {};
\node[draw,circle,inner sep=1pt,fill] at (1+6,3-12) {};

\draw[fill=gray!30] (0+3,0-16) -- (2+3,1-16) -- (1+3,3-16) -- cycle;
\node[draw,circle,inner sep=1pt,fill] at (0+3,0-16) {};
\node[draw,circle,inner sep=1pt,fill] at (2+3,1-16) {};
\node[draw,circle,inner sep=1pt,fill] at (1+3,3-16) {};
\draw[fill=gray!30] (0+3,0-21) -- (2+3,0-21) -- (0+3,4-21) -- cycle;
\node[draw,circle,inner sep=1pt,fill] at (0+3,0-21) {};
\node[draw,circle,inner sep=1pt,fill] at (2+3,0-21) {};
\node[draw,circle,inner sep=1pt,fill] at (0+3,4-21) {};

\draw[fill=gray!30] (0+6,0-21) -- (2+6,1-21) -- (0+6,4-21) -- cycle;
\node[draw,circle,inner sep=1pt,fill] at (0+6,0-21) {};
\node[draw,circle,inner sep=1pt,fill] at (2+6,1-21) {};
\node[draw,circle,inner sep=1pt,fill] at (0+6,4-21) {};

\draw[fill=gray!30] (0+3,0-26) -- (2+3,0-26) -- (1+3,4-26) -- cycle;
\node[draw,circle,inner sep=1pt,fill] at (0+3,0-26) {};
\node[draw,circle,inner sep=1pt,fill] at (2+3,0-26) {};
\node[draw,circle,inner sep=1pt,fill] at (1+3,4-26) {};

\draw[fill=gray!30] (0+6,0-26) -- (2+6,1-26) -- (1+6,4-26) -- cycle;
\node[draw,circle,inner sep=1pt,fill] at (0+6,0-26) {};
\node[draw,circle,inner sep=1pt,fill] at (2+6,1-26) {};
\node[draw,circle,inner sep=1pt,fill] at (1+6,4-26) {};
\draw[fill=gray!30] (0+9,0-12) -- (3+9,0-12) -- (0+9,3-12) -- cycle;
\node[draw,circle,inner sep=1pt,fill] at (0+9,0-12) {};
\node[draw,circle,inner sep=1pt,fill] at (3+9,0-12) {};
\node[draw,circle,inner sep=1pt,fill] at (0+9,3-12) {};

\draw[fill=gray!30] (0+13,0-12) -- (3+13,1-12) -- (0+13,3-12) -- cycle;
\node[draw,circle,inner sep=1pt,fill] at (0+13,0-12) {};
\node[draw,circle,inner sep=1pt,fill] at (3+13,1-12) {};
\node[draw,circle,inner sep=1pt,fill] at (0+13,3-12) {};

\draw[fill=gray!30] (0+9,0-16) -- (3+9,1-16) -- (1+9,3-16) -- cycle;
\node[draw,circle,inner sep=1pt,fill] at (0+9,0-16) {};
\node[draw,circle,inner sep=1pt,fill] at (3+9,1-16) {};
\node[draw,circle,inner sep=1pt,fill] at (1+9,3-16) {};

\draw[fill=gray!30] (0+13,0-16) -- (3+13,2-16) -- (1+13,3-16) -- cycle;
\node[draw,circle,inner sep=1pt,fill] at (0+13,0-16) {};
\node[draw,circle,inner sep=1pt,fill] at (3+13,2-16) {};
\node[draw,circle,inner sep=1pt,fill] at (1+13,3-16) {};
\draw[fill=gray!30] (0+9,0-21) -- (3+9,0-21) -- (0+9,4-21) -- cycle;
\node[draw,circle,inner sep=1pt,fill] at (0+9,0-21) {};
\node[draw,circle,inner sep=1pt,fill] at (3+9,0-21) {};
\node[draw,circle,inner sep=1pt,fill] at (0+9,4-21) {};

\draw[fill=gray!30] (0+13,0-21) -- (3+13,2-21) -- (0+13,4-21) -- cycle;
\node[draw,circle,inner sep=1pt,fill] at (0+13,0-21) {};
\node[draw,circle,inner sep=1pt,fill] at (3+13,2-21) {};
\node[draw,circle,inner sep=1pt,fill] at (0+13,4-21) {};

\draw[fill=gray!30] (0+9,0-26) -- (3+9,0-26) -- (1+9,4-26) -- cycle;
\node[draw,circle,inner sep=1pt,fill] at (0+9,0-26) {};
\node[draw,circle,inner sep=1pt,fill] at (3+9,0-26) {};
\node[draw,circle,inner sep=1pt,fill] at (1+9,4-26) {};

\draw[fill=gray!30] (0+13,0-26) -- (3+13,1-26) -- (1+13,4-26) -- cycle;
\node[draw,circle,inner sep=1pt,fill] at (0+13,0-26) {};
\node[draw,circle,inner sep=1pt,fill] at (3+13,1-26) {};
\node[draw,circle,inner sep=1pt,fill] at (1+13,4-26) {};

\draw[fill=gray!30] (0+9,0-31) -- (3+9,2-31) -- (1+9,4-31) -- cycle;
\node[draw,circle,inner sep=1pt,fill] at (0+9,0-31) {};
\node[draw,circle,inner sep=1pt,fill] at (3+9,2-31) {};
\node[draw,circle,inner sep=1pt,fill] at (1+9,4-31) {};
\draw[fill=gray!30] (0+17,0-21) -- (4+17,0-21) -- (0+17,4-21) -- cycle;
\node[draw,circle,inner sep=1pt,fill] at (0+17,0-21) {};
\node[draw,circle,inner sep=1pt,fill] at (4+17,0-21) {};
\node[draw,circle,inner sep=1pt,fill] at (0+17,4-21) {};

\draw[fill=gray!30] (0+22,0-21) -- (4+22,1-21) -- (0+22,4-21) -- cycle;
\node[draw,circle,inner sep=1pt,fill] at (0+22,0-21) {};
\node[draw,circle,inner sep=1pt,fill] at (4+22,1-21) {};
\node[draw,circle,inner sep=1pt,fill] at (0+22,4-21) {};

\draw[fill=gray!30] (0+17,0-26) -- (4+17,2-26) -- (0+17,4-26) -- cycle;
\node[draw,circle,inner sep=1pt,fill] at (0+17,0-26) {};
\node[draw,circle,inner sep=1pt,fill] at (4+17,2-26) {};
\node[draw,circle,inner sep=1pt,fill] at (0+17,4-26) {};

\draw[fill=gray!30] (0+22,0-26) -- (4+22,1-26) -- (1+22,4-26) -- cycle;
\node[draw,circle,inner sep=1pt,fill] at (0+22,0-26) {};
\node[draw,circle,inner sep=1pt,fill] at (4+22,1-26) {};
\node[draw,circle,inner sep=1pt,fill] at (1+22,4-26) {};

\draw[fill=gray!30] (0+17,0-31) -- (4+17,2-31) -- (1+17,4-31) -- cycle;
\node[draw,circle,inner sep=1pt,fill] at (0+17,0-31) {};
\node[draw,circle,inner sep=1pt,fill] at (4+17,2-31) {};
\node[draw,circle,inner sep=1pt,fill] at (1+17,4-31) {};

\draw[fill=gray!30] (0+22,0-31) -- (4+22,3-31) -- (1+22,4-31) -- cycle;
\node[draw,circle,inner sep=1pt,fill] at (0+22,0-31) {};
\node[draw,circle,inner sep=1pt,fill] at (4+22,3-31) {};
\node[draw,circle,inner sep=1pt,fill] at (1+22,4-31) {};

\draw[fill=gray!30] (0+17,0-36) -- (4+17,2-36) -- (2+17,4-36) -- cycle;
\node[draw,circle,inner sep=1pt,fill] at (0+17,0-36) {};
\node[draw,circle,inner sep=1pt,fill] at (4+17,2-36) {};
\node[draw,circle,inner sep=1pt,fill] at (2+17,4-36) {};

\foreach \x in {-1,...,27}{
    \foreach \y in {-37,...,1}{
        \node[cross=1.5pt] at (\x,\y) {};
    }
}
\draw[dashed] (-1,-0.5) -- (27,-0.5);
\draw[dashed] (-1,-2.5) -- (27,-2.5);
\draw[dashed] (-1,-8.5) -- (27,-8.5);
\draw[dashed] (-1,-16.5) -- (27,-16.5);
\draw[dashed] (0.5,1) -- (0.5,-37);
\draw[dashed] (2.5,1) -- (2.5,-37);
\draw[dashed] (8.5,1) -- (8.5,-37);
\draw[dashed] (16.5,1) -- (16.5,-37);
\end{tikzpicture}
    \caption{The triangles \(T \in \cS_{w_1,w_2}\) where \(0 \leq w_1 \leq w_2 \leq 4\).
    Columns (resp. rows) contain triangles with first (resp. second) width \(0\) to \(4\) from left to right (resp. top to bottom).
    When a width 0, degenerate triangle (see Section~\ref{sec:corollaries}) has multiple identical vertices these are denoted by concentric circles.}
    \label{fig:many_triangles}
\end{figure}

The idea of the proof is to use the fact that a triangle with first two widths \(w_1\) and \(w_2\) is equivalent to a subset of a \(w_1 \times w_2\) rectangle. 
It is straightforward to classify the collection of possible \(x\)-coordinates of vertices of a triangle in \([0,w_1] \times [0,w_2]\) up to affine equivalence.
The \(y\)-coordinates of each vertex are then integers in the range \([0,w_2]\) and we use facts about the width of the triangles to bound these integers.
By removing duplicates from the resulting list we obtain the set of triangles \(\cS_{w_1,w_2}\).
It remains to check that these triangles have the desired widths and are distinct.
This is made easier by the fact that when \(w_1<w_2\) we know that there is a unique (up to sign) direction in which the triangle has width \(w_1\).

The definition of first and second width extends to arbitrary dimension \(d\) by considering the minimal tuple of \(d\) widths of a polytope with respect to linearly independent dual vectors.
The methods of this proof extended to three dimensions towards a partial classification of lattice tetrahedra by their widths.
For more details see the sequel to this paper: \cite{tetrahedra}.

A consequence of Theorem~\ref{thm:main} is that we can completely classify lattice triangles by their affine automorphism group.
Another comes from the the fact that a lattice polygon is always equivalent to a subset of a rectangle with dimensions given by its first two widths.
Not only is this an integral part of the proof of Theorem~\ref{thm:main} but it also allows us to classify lattice triangles by the smallest square they are a subset of.
We can extend the classification to include degenerate triangles, that is multi-sets of three colinear lattice points.
If we do so then, up to affine equivalence, the number of lattice triangles which are a subset of \([0,n]^2\) is equal to the cardinality of the set \(nQ \cap \ZZ^4\) where \(Q\) is the four-dimensional simplex
\[
Q \coloneqq \frac12\conv\left(\left(1,0,0,0\right),\left(0,1,0,0\right),\left(0,0,1,0\right),(0,0,0,2),(-2,-2,-2,-2)\right).
\]
Additionally, the sequence counting the triangles in \([0,n]^2\) up to affine equivalence is given by a Hilbert function of a degree 8 hypersurface in \(\PP(1,1,1,2,2,2)\).

The Ehrhart polynomial is an important affine-invariant of a lattice polytope, counting the number of lattice points in its integral dilations.
In two dimensions this reduces to the study of pairs \((b,i)\), where \(b\) is the number of boundary points of a polygon and \(i\) is the number of interior points.
The possible pairs \((b,i)\) for lattice triangles can be plotted (see~Figure~\ref{fig:b_i_main_plot}), hinting at some beautiful combinatorial structures.
Hofscheier--Nill--\"Oberg~\cite{HNO} described infinitely many empty cones in this plot.
They also observed that points in the strips between these cones appear to form periodic patterns.
Using our classification of lattice triangles we will explain why these patterns appear and partially describe them.

In Section~\ref{sec:width} we will formally define the first two widths. We will prove that any lattice polygon with widths \(w\) and \(w'\) in linearly independent directions is equivalent to a subset of a \(w\times w'\) rectangle and as a result prove that \(\cT_{w_1,w_2}\) is equal to the set of triangles which are a subset of a \(w_1 \times w_2\) rectangle and no smaller up to equivalence.
In Section~\ref{sec:main_proof} we prove Theorem~\ref{thm:main}.
Propositions~\ref{prop:surjectivity}, \ref{prop:well-defined} and \ref{prop:injectivity} show that the map taking a triangle to its affine equivalence class is a well-defined, bijective map from \(\cS_{w_1,w_2}\) to \(\cT_{w_1,w_2}\).
In Section~\ref{sec:corollaries} we discuss corollaries of the main theorem.
We classify lattice triangles by their affine automorphism group, examine the generating functions of sequences arising in the classification and identify some related sequences which appear in Sloan's On-Line Encyclopedia of Integer Sequences \cite{OEIS}.
Finally, in Section~\ref{sec:ehrhart} we discuss the Ehrhart theory of lattice triangles.
We do so by studying the plot of the number of boundary and interior points of lattice triangles.
We provide an explanation for the periodic patterns which appear in strips of points in this plot and, by colouring the plot with the first and second widths of triangles realising each point, we provide an intuitive description of the triangles appearing in these strips.

\subsection*{Acknowledgements}
I would like to thank my supervisor Alexander Kasprzyk for introducing me to this problem and for much helpful guidance throughout.
I am also grateful to my supervisor Johannes Hofscheier for his valuable feedback and suggestions.
This research was supported in-part by the Heilbronn Institute for Mathematical Research.

\section{Width and rectangles}\label{sec:width}

Let \(N \cong \ZZ^d\) be a lattice, \(N^* \coloneqq \Hom(N,\ZZ) \cong \ZZ^d\) its dual lattice and \(N_{\RR} \coloneqq \RR \otimes_{\ZZ} N \cong \RR^d\) the real vector space containing \(N\).
For two pairs \((w_1,w_2)\) and \((w_1',w_2')\) we say that \((w_1,w_2) \leq_{lex} (w_1',w_2')\) either when \(w_1 < w_1'\) or when \(w_1=w_1'\) and \(w_2 \leq w_2'\).
This defines the lexicographic order on \(\ZZ^2\).
We can now define the first and second width.
\begin{definition}
For a lattice polytope \(P \subseteq N_\RR\) and a dual vector \(u \in N^*\) the \emph{width of \(P\) with respect to \(u\)} is
\[
\width_u(P) \coloneqq \max_{x \in P} \{u \cdot x\} - \min_{x \in P} \{u \cdot x\}.
\]
If \(d \geq 2\) then we can choose non-zero, linearly independent vectors \(u_1,u_2 \in N^*\) such that 
\[
(\width_{u_1}(P), \width_{u_2}(P))
\]
is minimal with respect to lexicographic order.
We call \(\width_{u_1}(P)\) the \emph{(first) width} of \(P\) written \(\width^1(P)\) and we call \(\width_{u_2}(P)\) the \emph{second width} of \(P\) written \(\width^2(P)\).
\end{definition}
It can be shown that these widths are the first two successive minima (in the sense of \cite{KannanLovasz}) of the polytope dual to \(P+(-P)\), where \(+\) denotes the Minkowski sum.

We need two more definitions for the following result.
A \emph{shear} is an affine map which preserves a line \(L\) in the plane.
It moves all other points in a fixed direction parallel to \(L\), by an amount proportional to the signed distance of the point from \(L\).
Such a map is said to be a \emph{shear about \(L\)}.
A lattice point is said to be \emph{primitive} if the greatest common divisor of all its coordinates is 1.

\begin{proposition}\label{prop:parallelogram}
Let \(Q\) be a lattice polygon with widths \(w\) and \(w'\) with respect to two linearly independent, primitive vectors in \(N^*\).
Then \(Q\) is equivalent to a subset of \([0,w] \times [0,w']\).
\end{proposition}
\begin{proof}
It suffices to prove this when \(w \leq w'\).
By a unimodular map we may assume that the width \(w\) is realised by \((1,0) \in N^*\).
Therefore, after a translation, we may assume that \(Q\) is a subset of a parallelogram \(P\) defined by
\[
P\coloneqq \{x \in \RR^2: (1,0)\cdot x \in [0,w], u\cdot x \in [d,d+w']\}
\]
for some integer \(d\) and some dual vector \(u\) linearly independent to \((1,0)\).
Say \(u=(u_x,u_y)\) then \(u_y \neq 0\) by linear independence.
After a shear about the \(y\)-axis, the vertices of \(P\) are
\[
\left(0,\frac{d}{u_y}\right),\left(0,\frac{d+w'}{u_y}\right),\left(w,\frac{d+w(ku_y-u_x)}{u_y}\right),\left(w,\frac{d+w'+w(ku_y-u_x)}{u_y}\right)
\]
for an integer \(k\).
We may assume that \(u_y>0\) otherwise we can replace \(u\) with \(-u\) and redefine \(d\) so that \(P\) is unchanged.
Therefore, we can choose \(k\) such that \(0 \leq ku_y-u_x < u_y\).
Now if \(Q'\) is the image of \(Q\) after this shear we have
\[
\width_{(0,1)}(Q') \leq \frac{w'+w(ku_y-u_x)}{u_y} \leq w'\frac{1+ku_y-u_x}{u_y} \leq w'.
\]
Possibly after a translation \(Q'\) is the desired subset of \([0,w] \times [0,w']\).
This relies on the fact that \(Q\) is a lattice polygon and therefore has integral width.
\end{proof}

Since width is an affine invariant \(\cT_{w_1,w_2}\) (Definition~\ref{def:main}), is well defined, but there is further motivation for this definition.
Given a lattice triangle \(T\), there is a natural lattice rectangle with dimensions \(\width_{(1,0)}(T) \times \width_{(0,1)}(T)\) which \(T\) is a subset of.
Another triangle equivalent to \(T\) may have a different associated rectangle so it is interesting to consider what the ``smallest'' rectangle containing \(T\) is up to equivalence.
One can formalise the idea of ``smallest'' using lexicographic order on rectangles or by requiring a triangle to be equivalent to no subset of a sub-rectangle.
Formally this is defined as follows.
\begin{definition}
Let \(\cT^{lex}_{w_1,w_2}\) be the set of lattice triangles \(T\) (up to affine equivalence) which are equivalent to a subset of \([0,w_1] \times [0,w_2]\) such that for all \((w_1',w_2') <_{lex} (w_1,w_2)\), \(T\) is not equivalent to a subset of \([0,w_1']\times [0,w_2']\).

Let \(\cT^{sub}_{w_1,w_2}\) be the set of lattice triangles \(T\) (up to affine equivalence) which are equivalent to a subset of \([0,w_1]\times[0,w_2]\) such that for all lattice points \((w_1',w_2') \in [0,w_1] \times [0,w_2]\) if \(T\) is equivalent to a subset of \([0,w_1'] \times [0,w_2']\) then \((w_1',w_2') = (w_1,w_2)\).
\end{definition}
For a lattice triangle \(T\) there is a unique pair of integers \((w_1,w_2)\) such that the equivalence class of \(T\) is in \(\cT^{lex}_{w_1,w_2}\).
However, while there is at least one pair \((w_1,w_2)\) such that the equivalence class of \(T\) is in \(\cT^{sub}_{w_1,w_2}\) it is not immediate that there is only one such pair.
In fact, if we allow \(w_2 < w_1\) there is more than one for infinitely many triangles.
The following proves that if we require \(w_1 \leq w_2\) this pair is unique and that both of these sets are equal to \(\cT_{w_1,w_2}\).

\begin{proposition}\label{prop:lex_eq_width}
The sets \(\cT_{w_1,w_2}\), \(\cT^{lex}_{w_1,w_2}\) and \(\cT^{sub}_{w_1,w_2}\) are equal.
\end{proposition}
\begin{proof}
If the equivalence class of \(T\) is in \(\cT_{w_1,w_2}\) then there are linearly independent, primitive directions with respect to which it has widths \(w_1\) and \(w_2\).
Therefore, by Proposition~\ref{prop:parallelogram} \(T\) is equivalent to a subset of \([0,w_1] \times [0,w_2]\).
If \(T\) were equivalent to a subset of some \(w_1' \times w_2'\) rectangle where \((w_1',w_2') <_{lex} (w_1,w_2)\) this would contradict the widths of \(T\) so \(\cT_{w_1,w_2} \subseteq \cT_{w_1,w_2}^{lex}\).
If \(T\) were equivalent to a subset of some \(w_1'\times w_2'\) rectangle where \((w_1',w_2') \in [0,w_1] \times [0,w_2] \cap \ZZ^2\) then its widths would force \((w_1',w_2')=(w_1,w_2)\) so \(\cT_{w_1,w_2} \subseteq \cT_{w_1,w_2}^{sub}\).

If the equivalence class of \(T\) is in \(\cT^{lex}_{w_1,w_2}\) then its widths are lexicographically at most \(w_1\) and \(w_2\).
If we had \((\width^1(T), \width^2(T)) <_{lex} (w_1,w_2)\) then by Proposition~\ref{prop:parallelogram} \(T\) would be equivalent to a subset of \([0,\width^1(T)] \times [0,\width^2(T)]\) contradicting the definition of \(\cT_{w_1,w_2}^{lex}\).
Therefore, \(T\) has widths \(w_1\) and \(w_2\) and \(\cT_{w_1,w_2} = \cT_{w_1,w_2}^{lex}\).

If the equivalence class of \(T\) is in \(\cT_{w_1,w_2}^{sub}\) then its widths are lexicographically at most \(w_1\) and \(w_2\).
We may assume that \(T \subseteq [0,w_1] \times [0,w_2]\).
If it had width \(w_2'\) in some direction linearly independent to \((1,0)\) and \(w_2' < w_2\) then by Proposition~\ref{prop:parallelogram}, \(T\) is equivalent to a subset of \([0,w_1] \times [0,w_2']\) which contradicts the definition of \(\cT_{w_1,w_2}^{sub}\).
If it had width \(w_1' < w_1\) with respect to \((1,0)\) then, possibly after a translation, it would be a subset of \([0,w_1-1]\times [0,w_2]\) which also contradicts the definition of \(\cT_{w_1,w_2}^{sub}\).
Therefore, \(T\) has widths \(w_1\) and \(w_2\) and \(\cT_{w_1,w_2} = \cT_{w_1,w_2}^{sub}\).
\end{proof}

From now on we will freely use any of these definitions to describe the elements of \(\cT_{w_1,w_2}\).

\section{Proof of Theorem~\ref{thm:main}}
\label{sec:main_proof}

The following four results form the proof of Theorem~\ref{thm:main} by showing that the map taking a triangle to its affine equivalence class is a well-defined bijection from \(\cS_{w_1,w_2}\) to \(\cT_{w_1,w_2}\).
First we prove surjectivity.
\begin{proposition}\label{prop:surjectivity}
Let \(T\) be a lattice triangle with first and second width \(w_1\) and \(w_2\) respectively.
Then there exists a triangle \(T' \in \cS_{w_1,w_2}\) which is affine equivalent to \(T\).
\end{proposition}
\begin{proof}

By Proposition~\ref{prop:lex_eq_width} we may assume that \(T\) is a subset of \([0,w_1] \times [0,w_2]\).
We may further assume that there is a vertex of \(T\) contained in each edge of the rectangle otherwise it would be a subset of a smaller rectangle.
Consider the three-point set we obtain by projecting the vertices of \(T\) onto the first coordinate.
This set is \(\{0,x_2,w_1\}\) for some integer \(x_2 \in [0,w_1]\).
By a reflection in the line \(x=w_1/2\) we may assume that \(x_2 \leq w_1/2\).

Now, for some integers \(y_0,y_1,y_2 \in [0,w_2]\)
\[
T = \conv((0,y_0),(w_1,y_1),(x_2,y_2)).
\]
Since \(T\) has vertices in each edge of the rectangle one of these \(y\)-coordinates must be \(0\) and one must be \(w_2\).
Suppose, towards a contradiction, that \(0 < y_2 < w_2\).
Then, possibly after a reflection in the line \(y=w_2/2\), we can assume that \(y_0=0\) and \(y_1=w_2\) so the affine map \((x,y) \mapsto (x,y-x)\) takes \(T\) to 
\[
\conv((0,0),(w_1,w_2-w_1),(x_2,y_2-x_2)).
\]
However, this is a subset of a smaller rectangle which contradicts the widths of \(T\), so we may assume that \(y_2=0\) or \(w_2\).
By a reflection in the line \(y=w_2/2\) we assume that \(y_2=w_2\).

Now, for integers \(y_0, y_1 \in [0,w_2]\) one of which is zero, we have
\[
T = \conv((0,y_0),(w_1,y_1),(x_2,w_2))
\]
and we need to consider three different cases corresponding to the three different types of triangle in \(\cS_{w_1,w_2}\).
We will prove each of the following facts:
\begin{enumerate}
    \item If \(x_2=0\) then we may assume \(T\) is of type (A),
    \item If \(x_2>0\) and \(y_0=0\) then we may assume \(T\) is of type (B)
    \item If \(x_2>0\) and \(y_0>0\) then we may assume \(T\) is of type (C).
\end{enumerate}

(1) Suppose \(x_2=0\) and consider the image of the vertices of \(T\) under \((1,1)\):
\[
\{y_0,w_1+y_1,w_2\}.
\]
By the second width of \(T\), this must not be a subset of \([1,w_2]\).
We know that \(y_1=0\) or \(y_0=0\).
If \(y_1=0\) then \(w_1+y_1\) and \(w_2\) are both contained in \([1,w_2]\) so we must have \(y_0=0\) too so we always have \(y_0=0\).
By a shear about the \(y\)-axis we may assume that \(y_1 < w_1\).
Define \(y_1'\) to be \((w_2-y_1 \mod w_1)\).
Notice that, since \(0 \leq y_1 < w_1\), we have \(y_1 = (w_2-y_1' \mod w_1)\).
If \(y_1 \leq y_1'\) then \(T\) is of type (A).
Otherwise, let \(k\) be the integer such that \(y_1' = w_2-y_1 + kw_1\) then the map \((x,y)\mapsto (x,w_2-y+kx)\) takes \(T\) to the following triangle of the form (A):
\[
T' = \conv((0,0),(w_1,y_1'),(0,w_2)).
\]

(2) If \(x_2>0\) and \(y_0=0\) then consider the image of \(T\) under (-1,1):
\[
\{0,y_1-w_1,w_2-x_2\}.
\]
Due to the second width of \(T\) this must not be a subset of a line segment of length less than \(w_2\) which means that \((w_2-x_2) - (y_1-w_1)\) must be at least \(w_2\).
We can rearrange this to show that \(0 \leq y_1 \leq w_1-x_2\).
If additionally \(w_1=w_2\) then we have the following two triangles which are equivalent, under the map exchanging \(x\)- and \(y\)-coordinates.
\[
\conv((0,0),(w_1,y_1),(x_2,w_1)), \quad \conv((0,0),(w_1,x_2),(y_1,w_1))
\]
We choose the triangle with the smaller \(x\)-coordinates and so may assume that \(y_1 \geq x_2\).

(3) If \(x_2>0\) and \(y_0>0\) then we must have \(y_1=0\).
If \(x_2=w_1/2\) then a reflection takes us to the previous case so we may assume \(x_2 < w_1/2\).
Consider the image of \(T\) under \((1,1)\):
\[
\{y_0,w_1,w_2+x_2\}.
\]
We know that \(y_0\) and \(w_1\) are both less than \(w_2+x_2\) so to prevent this fitting in a line segment of length less than \(w_2\) we must have either \(y_0\) or \(w_1\) must less than or equal to \(x_2\).
It is fixed that \(x_2 < w_1\) so we must have \(y_0 \leq x_2\)
If \(y_0=x_2\) then, under the map \((x,y) \mapsto (x,y-x_2 +x)\), \(T\) is equivalent to \(\conv((0,0), (w_1,w_1-x_2), (x_2,w_2))\) which is included in case (2) so we may assume that \(y_0<x_2\).
If \(w_1=w_2\) then the map \((x,y) \mapsto (y,w_1-x)\) takes \(T\) to \(\conv((0,0),(y_0,w_1),(w_1,w_1-x_2))\) which is included in case (2) so we may assume case (3) only occurs when \(w_1<w_2\).
This shows that \(T\) is equivalent to a triangle in \(\cS_{w_1,w_2}\).
\end{proof}

We now show that the widths of triangles in \(\cS_{w_1,w_2}\) are \(w_1\) and \(w_2\) which we will use to show that the map taking a triangle to its equivalence class is a map from \(\cS_{w_1,w_2}\) to \(\cT_{w_1,w_2}\).

\begin{proposition}\label{prop:well-defined}
For \(T \in \cS_{w_1,w_2}\), the first and second widths of \(T\) are \(w_1\) and \(w_2\) respectively.
\end{proposition}

\begin{proof}
The triangles in \(\cS_{w_1,w_2}\) fall into one of the following types regardless of whether \(w_1=w_2\) or \(w_1<w_2\).
Therefore, it suffices to prove the result for each of the following triangles for all positive integers \(w_1 \leq w_2\).
\begin{enumerate}
    \item[(A)] \(T=\conv((0,0),(w_1,y_1),(0,w_2))\) where \(0 \leq y_1 \leq (w_2-y_1 \mod w_1)\)
    \item[(B)] \(T=\conv((0,0),(w_1,y_1),(x_2,w_2))\) where \(0 < x_2 \leq \frac{w_1}{2}\) and \(0 \leq y_1 \leq w_1-x_2\)
    \item[(C)] \(T=\conv((0,y_0),(w_1,0),(x_2,w_2))\) where \(1 < x_2 < \frac{w_1}{2}\), \(0 < y_0 < x_2\) and \(w_1<w_2\).
\end{enumerate}
These triangles all have widths \(w_1\) and \(w_2\) with respect to \((1,0)\) and \((0,1)\) respectively so it suffices to show that they have width at least \(w_2\) in all directions linearly independent to \((1,0)\)
By  the proof of Proposition~\ref{prop:parallelogram}, it suffices to show that under any shear \(\varphi\) about the \(y\)-axis \(\width_{(0,1)}(\varphi(T)) \geq w_2\). We do this in each of the three cases.

(A) This case it is immediate since all shears about the \(y\)-axis preserve \((0,0)\) and \((0,w_2)\).

(B) A shear about the \(y\)-axis can take this case to
\[
\conv((0,0),(w_1,y_1+kw_1),(x_2,w_2+kx_2))
\]
for integers \(k\).
Suppose for contradiction that this is has width less than \(w_2\) with respect to \((0,1)\).
Then we must have \(w_2+kx_2<w_2\) so \(k\) is negative.
We also need \(w_2+kx_2-y_1-kw_1 < w_2\) and so \(\frac{-y_1}{w_1-x_2} < k\).
However, we know that \(w_1-x_2 \geq y_1\) so no such integer \(k\) exists.

(C) A shear about the \(y\)-axis can take this case to
\[
\conv((0,y_0),(w_1,kw_1),(x_2,w_2+kx_2))
\]
for integers \(k\).
Suppose for contradiction that this has width less than \(w_2\) with respect to \((0,1)\).
Then we must have \(w_2+kx_2-kw_1<w_2\) so \(k\) is positive.
We also need \(w_2+kx_2-y_0 < w_2\) and so \(k < \frac{y_0}{x_2}\) but we know that \(y_0<x_2\) so no such integer \(k\) exists.
\end{proof}

The following lemma will be used in the proof of distinctness of the triangles in \(\cS_{w_1,w_2}\).

\begin{lemma}\label{lemma:x_2_min}
Let \(T = \conv((0,y_0),(w_1,y_1),(x_2,w_2))\) be a triangle in \(\cS_{w_1,w_2}\).
Let \(u\) be a dual vector such that the image of the vertices of \(T\) under \(u\) is \(\{0,x_2',w_1\}\) for some integer \(x_2' \in [0,w_1]\) then \(x_2' \geq x_2\).
In other words, \(x_2\) is minimal.
\end{lemma}
\begin{proof}
When \(x_2=0\) this is immediate.
When \(w_1<w_2\) it follows from the facts that \(\width_u(T)=w_1\) for a unique (up to sign) choice of \(u\) and that \(x_2 \leq w_1/2\).
Therefore, we need only prove this for triangles of the form (B) when \(w_1=w_2\).
That is, 
\[
T = \conv((0,0),(w_1,y_1),(x_2,w_1)) \quad \text{with} \quad 0 < x_2 \leq w_1/2 \quad \text{and} \quad x_2 \leq y_1 \leq w_1-x_2.
\]

Let \(u=(u_x,u_y)\) be a dual vector linearly independent to \((1,0)\), we may assume that \(u_y \geq 1\). The image of the vertices of \(T\) under \(u\) is
\[
\{0, u_xw_1+u_yy_1, u_xx_2+u_yw_1\}.
\]
Pick \(u\) such that this is equivalent to \(\{0,x_2',w_1\}\) for some \(x_2' \in [0,w_1]\).
This places strong restrictions on \(u\).
The difference between each pair of elements in this set must be at most \(w_1\).
In particular \(u_xx_2+u_yw_1\leq w_1\) and \(u_xw_1+u_yy_1 \geq -w_1\).
These can be rearranged into
\begin{align*}
-(u_yy_1+w_1)/w_1 \leq u_x \leq w_1(1-u_y)/x_2
\end{align*}
By the conditions on \(T\) we know that \(w_1/x_2 \geq 2\) and \(y_1/w_1 < 1\) so we can further show that
\begin{align}\label{eq:u_x_bounds}
-1 -u_y < u_x \leq 2(1-u_y).
\end{align}
In particular \(u_y < 3\).
Substituting \(u_y=1\) and \(u_y=2\) into \eqref{eq:u_x_bounds} and considering the possible integers \(u_x\) in each case shows that \(u\) is equal to \((0,1)\), \((-1,1)\) or \((-2,2)\).
The image of the vertices of \(T\) under each of these are 
\[
\{0,y_1,w_1\},\quad \{0,y_1-w_1,w_1-x_2\} \quad \text{and} \quad \{0,2y_1-2w_1,2w_1-2x_2\}
\]
respectively.
The properties of the coordinates of \(T\) allow us to order the elements of each of these sets: \(0 < y_1 < w_1\) and \(y_1-w_1 < 0 < w_1-x_2\).
This allows us to identify which point in each set is sent to \(x_2'\) under the equivalence with \(\{0,x_2',w_1\}\).
We see that \(x_2' = y_1\), \(w_1-y_1\), \(w_1-x_2\), \(2(w_1-x_2)\) or \(2(w_1-y_1)\).
For any of these \(x_2' \geq x_2\) as desired.
\end{proof}

Now we can prove affine distinctness of the triangles in \(\cS_{w_1,w_2}\).

\begin{proposition}\label{prop:injectivity}
The triangles in \(\cS_{w_1,w_2}\) are all distinct under affine unimodular maps.
\end{proposition}
\begin{proof}
Let \(T\) and \(T'\) be equivalent triangles in \(\cS_{w_1,w_2}\).
Let the variables associated to \(T\) and \(T'\) be denoted by \(y_0,y_1,x_2\) and \(y_0',y_1',x_2'\) respectively.
By Lemma~\ref{lemma:x_2_min} \(x_2=x_2'\).
Therefore, both are of the form (A) or neither are.
We will show in each of the following four facts.
\begin{enumerate}
    \item If \(T\) and \(T'\) are both of type (A) then \(T=T'\)
    \item If \(T\) and \(T'\) are both of type (B) then \(T=T'\)
    \item If \(T\) and \(T'\) are both of type (C) then \(T=T'\)
    \item If \(T\) is of type (B) and \(T'\) is of type (C) then we have a contradiction.
\end{enumerate}

(1) Either \(y_1=y_1'=0\) and \(w_1=w_2\) or they each have a unique edge of lattice length \(w_2\).
Therefore, either \(T=T'\) or the affine map taking \(T\) to \(T'\) preserves the line segment from \((0,0)\) to \((0,w_2)\).
The reduces us to maps of the form 
\[
(x,y) \mapsto (x,y+kx) \quad \text{and} \quad(x,y) \mapsto (x,w_2-y+kx).
\]
for integers \(k\).
Since \(0 \leq y_1,y_1' < w_1\) if a map of the first from takes \(T\) to \(T'\) then \(y_1=y_1'\) and \(T=T'\).
If a map of the second form takes \(T\) to \(T'\) then \(y_1 \leq (w_2-y_1 \mod w_1) = y_1'\) and symmetrically \(y_1' \leq y_1\) so again \(T=T'\).

(2) The normalised volumes of \(T\) and \(T'\) are \(w_1w_2-y_1x_2\) and \(w_1w_2-y_1'x_2\) which must be equal so \(y_1=y_1'\) and \(T=T'\).

(3) The normalised volumes of \(T\) and \(T'\) are \(w_1w_2-w_1y_0+y_0x_2\) and \(w_1w_2-w_1y_0'+y_0'x_2\) which must be equal so \(y_0=y_0'\) and \(T=T'\).

(4) Since \(T'\) is of type (C) we know that \(0 < x_2 < w_1/2\) and \(w_1<w_2\).
By \(w_1<w_2\) we see that the dual vector \(u\) such that \(\width_u(T')=w_1\) is unique.
We can use it to distinguish the three vertices of \(T\) and \(T'\) by their images under \(u\).
In this way we see there is only one order in which to map the vertices of \(T\) to the vertices of \(T'\) that is \((0,0)\), \((w_1,y_1)\) and \((x_2,w_2)\) map to \((0,y_0')\), \((w_1,0)\) and \((x_2,w_2)\) respectively.
Thus the following matrix must be unimodular
\[
\begin{pmatrix}
w_1 & x_2\\ -y_0' & w_2-y_0'
\end{pmatrix}
\begin{pmatrix}
w_1 & x_2\\ y_1 & w_2
\end{pmatrix}^{-1}
=
\begin{pmatrix}
1 & 0 \\ \frac{y_1y_0'-y_0'w_2-y_1w_2}{w_1w_2-x_2y_1} & 1 
\end{pmatrix}.
\]
Since \(w_2 > y_1\) we must have \(y_1y_0'-y_0'w_2-y_1w_2 < 0\) so for the matrix to have integral entries we must have \(y_1w_2+y_0'w_2-y_1y_0' \geq w_1w_2-x_2y_1\).
Since \(y_1 \leq w_1-x_2\) the right hand side is at least \(w_2(x_2+y_1) - x_2y_1\).
Cancelling terms and dividing by \(w_2-y_1\) gives \(y_0' \geq x_2\) which is the desired contradiction.
\end{proof}

We bring together the results of this section to classify and count triangles by their widths:

\begin{proof}[Proof of Theorem~\ref{thm:main}]
Proposition~\ref{prop:well-defined} shows that the map taking a triangle to its affine equivalence class is a map from \(\cS_{w_1,w_2}\) to \(\cT_{w_1,w_2}\).
By Propositions~\ref{prop:surjectivity} and \ref{prop:injectivity} this map is bijective.

To calculate the cardinality of \(\cT_{w_1,w_2}\) we need only count the triangles in \(\cS_{w_1,w_2}\).
For triangles of type (A) we need to compute the number of integers \(y_1 \in [0,w_1)\) such that \(y_1 \leq (w_2-y_1 \mod w_1)\).
Pick \(q,r \in \ZZ\) with \(0\leq r<w_1\) such that \(w_2 = qw_1+r\).
Then \(y_1 \leq (w_2-y_1 \mod w_1)\) if and only if \(0 \leq y_1 \leq \frac{r}{2}\) or \(r<y_1\leq \frac{r+w_1}{2}\) which can be seen by considering the plot of \((y_1,y_1)\) and \((y_1,w_2-y_1 \mod w_1)\) for \(y_1 \in [0,w_1)\).
For any integer \(a\) the number of points in \([0,\frac{a}{2}]\cap \ZZ\) is \(\lceil \frac{a+1}{2} \rceil\) so the number of \(y_1\) satisfying one of the above is
\[
\left\lceil\frac{r+1}{2}\right\rceil + \left\lceil\frac{w_1-r+1}{2}\right\rceil - 1.
\]
If we substitute in \(r = w_2-qw_1\) and consider cases for \(w_1, w_2\) and \(q\) odd and even it can be shown that this is \(\lceil \frac{w_1}{2}\rceil\) when \(w_2\) is odd and \(\lceil \frac{w_1+1}{2} \rceil\) when \(w_2\) is even.

For triangles of type (B) and (C) we need only calculate 
\[
\sum_{x_2=1}^{\lfloor w_1/2\rfloor} (w_1-x_2+1) + \sum_{x_2=2}^{\lfloor (w_1-1)/2 \rfloor} (x_2-1) = \begin{cases}
    \frac{w_1^2}{2} - \frac{w_1}{2}+1 &\text{if \(w_1\) even}\\
    \frac{w_1^2}{2}-\frac{w_1}{2} &\text{if \(w_1\) odd}\\
    \end{cases}
\]
when \(w_1<w_2\) and
\[
\sum_{x_2=1}^{\lfloor w_1/2 \rfloor} (w_1-2x_2+1)=\begin{cases}
    \frac{w_1^2}{4} &\text{if \(w_1\) even}\\
    \frac{w_1^2}{4} -\frac{1}{4} &\text{if \(w_1\) odd}\\
    \end{cases}
\]
when \(w_1=w_2\).
Combining these sums and separating odd and even cases gives the result.
\end{proof}

\section{Corollaries}
\label{sec:corollaries}

A consequence of Theorem~\ref{thm:main} is that we have defined a normal form for lattice triangles from which their first two widths can be read.
This normal form is compatible with scaling in the sense that for a positive integer \(\lambda\) and a triangle \(T \in \cS_{w_1,w_2}\) we have \(\lambda T \in \cS_{\lambda w_1,\lambda w_2}\).
We can also read the affine automorphism group of a triangle from this normal form as shown in the following corollary.
Let \(S_3\) denote the group of permutations of \(\{0,1,2\}\), written in cycle notation, and \(\Aut(T)\) denote the group of affine maps which map \(T\) to itself.

\begin{corollary}
Let \(T \in \cS_{w_1,w_2}\) then 
\begin{itemize}
    \item \(\Aut(T) \cong S_3\) if and only if one of the following hold
    \begin{itemize}
        \item \(w_1=w_2\) and \(T=((0,0),(w_1,0),(0,w_1))\) or
        \item \(w_1=w_2\) and \(T=((0,0),(w_1,\frac{w_1}{2}),(\frac{w_1}{2},w_1))\)
    \end{itemize}
    \item \(\Aut(T) \cong \langle(012)\rangle\) if and only if the following holds
    \begin{itemize}
        \item \(w_1=w_2\) and \(T=\conv((0,0), (w_1,y_1), (w_1-y_1,w_1))\) such that \(y_1 \neq \frac{w_1}{2}\)
    \end{itemize}
    \item \(\Aut(T) \cong \langle(01)\rangle\) if and only if one of the following hold
    \begin{itemize}
        \item \(T=\conv((0,0),(w_1,y_1),(0,w_2))\) such that \(y_1 \equiv (w_2-y_1 \mod w_1)\) and either \(y_1>0\) or \(w_1<w_2\),
        \item \(T=\conv((0,0),(w_1,0),(\frac{w_1}{2},w_2))\)
        \item \(T=\conv((0,0),(w_1,\frac{w_1}{2}),(\frac{w_1}{2},w_2))\) and \(w_1<w_2\),
        \item \(T=\conv((0,0),(w_1,y_1),(y_1,w_1))\), \(w_1=w_2\) and \(2y_1 < w_1\)
    \end{itemize}
    \item \(\Aut(T) \cong \{\iota\}\) otherwise.
\end{itemize}
\end{corollary}
\begin{proof}
We consider each of the three types of triangle in \(\cS_{w_1,w_2}\) and find conditions for them to have each automorphism group.

Triangles \(T\) of type (A) have automorphism group \(S_3\) when \(w_1=w_2\) and \(y_1=0\) since this is just a dilation of the standard simplex \(\conv((0,0),(1,0),(0,1))\).
Otherwise any automorphism of \(T\) must preserve the edge from \((0,0)\) to \((0,w_2)\) so we are reduced to maps of the form \((x,y) \mapsto (x,y+kx)\) and \((x,y) \mapsto (x,w_2-y+kx)\) for integers \(k\).
The first of these maps can only map \(T\) to itself if \(k=0\) which is just the identity map.
The second can only map \(T\) to itself if \(y_1 = (w_2 - y_1 \mod w_1)\).
So \(T\) has automorphism group isomorphic to \(\langle(12)\rangle\) in this case and trivial automorphism group otherwise.

For a triangle \(T\) of type (B) let \(\varphi\) be an affine map taking \(T\) to itself.
It is defined by multiplication by a unimodular matrix \(U \in \GL_2(\ZZ)\) followed by a translation \(t \in \ZZ^2\).
Let \(v_0=(0,0)\), \(v_1=(w_1,y_1)\) and \(v_2=(x_2,w_2)\) be the vertices of \(T\) then the set of numbers \(\{(1,0) \cdot \varphi(v_i) : i=0,1,2\}\) is equal to \(\{0,x_2,w_1\}\).
This means that the set of numbers \((1,0)U(v_i)\) is equivalent to \(\{0,x_2,w_1\}\) and so \(T\) has width \(w_1\) with respect to the vector \((1,0) U\).
If \(w_1<w_2\) this forces the first row of \(U\) to be \((\pm1,0)\).
Otherwise, by the proof of Lemma~\ref{lemma:x_2_min} the first row of \(U\) must be \((\pm1,0)\), \((0,\pm1)\) or \((\pm1,\mp1)\).
Therefore, for some integer \(k\), \(U\) is one of the following
\[
\begin{pmatrix}
    1 & 0 \\
    k & \pm1
\end{pmatrix},
\begin{pmatrix}
    -1 & 0 \\
    k & \pm1
\end{pmatrix},
\begin{pmatrix}
    0 & 1 \\
    \pm1 & k
\end{pmatrix},
\begin{pmatrix}
    0 & -1 \\
    \pm1 & k
\end{pmatrix},
\begin{pmatrix}
    1 & -1 \\
    \pm1-k & k
\end{pmatrix},
\begin{pmatrix}
    -1 & 1 \\
    \pm1-k & k
\end{pmatrix}.
\]
The image of \(T\) under each of these maps, in the same order, is
\begin{enumerate}
    \item[] \(T_1=\conv((0,0),(w_1,kw_1\pm y_1),(x_2,kx_2\pm w_2))\)
    \item[] \(T_2=\conv((0,0),(-w_1,kw_1\pm y_1),(-x_2,kx_2\pm w_2))\)
    \item[] \(T_3=\conv((0,0),(y_1,\pm w_1+ky_1),(w_2,\pm x_2+kw_2))\)
    \item[] \(T_4=\conv((0,0),(-y_1,\pm w_1+ky_1),(-w_2,\pm x_2+kw_2))\)
    \item[] \(T_5=\conv((0,0),(w_1-y_1,(\pm1-k)w_1+ky_1),(x_2-w_2,(\pm1-k)x_2+kw_2))\)
    \item[] \(T_6=\conv((0,0),(y_1-w_1,(\pm1-k)w_1+ky_1),(w_2-x_2,(\pm1-k)x_2+kw_2))\).
\end{enumerate}
These must be translations of \(T\).
In each case, sorting the vertices by the size of their \(x\)-coordinates indicates the permutation of the vertices of \(T\) they correspond to.
This allows us to find the translation vector \(t\).
The sign of the permutation indicates weather the determinant of \(U\) is positive or negative.
This is enough to explicitly compute \(k\) and obtain a collection of restrictions on the variables.
In the following we execute this procedure in each case:

If \(T_1\) is a translation of \(T\) it comes from the identity permutation, so \(t=(0,0)\) and the determinant of \(U\) is 1.
By equating the vertices of \(T_1+t\) and \(T\) we can show that \(k=0\) so \(\varphi\) is the identity map.

If \(T_2\) is a translation of \(T\) it comes from the permutation \((01)\) of the vertices, so \(t=(w_1,y_1)\) and the determinant of \(U\) is \(-1\).
By equating the vertices of \(T_2+t\) and \(T\) we can show that \(2x_2=w_1\) and \(k=-2y_1/w_1\) which is only an integer if \(y_1=0\) or \(y_1=w_1/2\).
Then \(k=0\) or \(-1\) each of which lead to valid automorphisms.

If \(T_3\) is a translation of \(T\) it comes from the permutation \((12)\) of the vertices, so \(t=(0,0)\) and the determinant of \(U\) is \(-1\).
By equating the vertices of \(T_3+t\) and \(T\) we can show that \(y_1=x_2\), \(w_1=w_2\) and \(k=0\) which leads to a valid automorphism of \(T\).

If \(T_4\) is a translation of \(T\) it comes from the permutation \((012)\) of the vertices, so \(t=(w_1,y_1)\) and the determinant of \(U\) is \(1\).
By equating the vertices of \(T_4+t\) and \(T\) we can show that \(w_1=w_2=x_2+y_1\) and \(k=-1\) which leads to a valid automorphism of \(T\).

If \(T_5\) is a translation of \(T\) it comes from the permutation \((02)\) of the vertices, so \(t=(x_2,w_2)\) and the sign of \(U\) is \(-1\). 
By equating the vertices of \(T_5+t\) and \(T\) we can show that \(w_1=w_2=2x_2=2y_1\) and \(k=-1\) which gives a valid automorphism of \(T\).

If \(T_6\) is a translation of \(T\) it comes from the permutation \((021)\) of the vertices, so \(t=(x_2,w_2)\) and the determinant of \(U\) is \(1\). 
By equating the vertices of \(T_6+t\) and \(T\) we can show that \(w_1=w_2=x_2+y_1\) and \(k=0\) which gives a valid automorphism of \(T\).

Notice that there is overlap between some of these cases.
For example if \(T_5\) is a translation of \(T\) then \(T\) automatically satisfies the conditions for all other \(T_i\) to be a translation of \(T\) so \(T\) has automorphism group \(S_3\).
By considering the ways in which \(T\) can have automorphism group isomorphic to each subgroup of \(S_3\) the above give the rest of the triangles in the corollary with non-trivial automorphism group.

Triangles of type (C) all have trivial automorphism group.
We show this using a very similar method to the type (B) triangles above.
For a triangle \(T\) of type (C) let \(\varphi\) be an affine map taking \(T-(0,y_0)\) to itself (where we subtract \((0,y_0)\) so we may assume \(T\) has a vertex at the origin).
This is defined by a unimodular matrix \(U\) and translation vector \(t\).
Since \(w_1<w_2\) the first row of \(U\) must be \((\pm1,0)\) so, for some integer \(k\), \(U\) is one of the following
\[
\begin{pmatrix}
    1 & 0 \\ k & \pm 1
\end{pmatrix}, \quad 
\begin{pmatrix}
    -1 & 0 \\ k & \pm 1
\end{pmatrix}.
\]
The image of \(T-(0,y_0)\) under each of these maps is
\begin{itemize}
    \item[] \(T_7 = \conv((0,0),(w_1,kw_1\mp y_0),(x_2,kx_2\pm (w_2-y_0)))\) or
    \item[] \(T_8 = \conv((0,0),(-w_1,kw_1\mp y_0),(-x_2,kx_2\pm (w_2-y_0)))\).
\end{itemize}

If \(T_7\) is a translation of \(T-(0,y_0)\) it comes from the identity permutation so \(t=(0,0)\) and the determinant of \(U\) is \(1\).
By equating the vertices of \(T_7+t\) and \(T-(0,y_0)\) we see that \(k=0\) so this is the identity map.

If \(T_8\) is a translation of \(T\) it comes from the permutation \((01)\) so \(t=(w_1,-y_0)\) and the determinant of \(U\) is \(-1\).
By equating the \(x\)-coordinates of \(T_8+t\) and \(T-(0,y_0)\) we see that \(w_1-x_2 = x_2\) so \(x_2=w_1/2\) which contradicts the definition of type (C) triangles.
\end{proof}

So far we have restricted ourselves to non-degenerate triangles, that is triangles with non-zero volume.
However, we can extend our definitions and proofs to multi-sets containing three colinear points and call these points the vertices of a degenerate triangle.
Let the widths of such a set be the widths of its convex hull.
It has first and second widths \(0\) and \(l\) where \(l\) is its lattice length.
In this sense we can see that there are \(\left\lceil(w_2+1)/2\right\rceil\) triangles with first width \(0\) and second width \(w_2\) up to affine equivalence for all \(w_2 \geq 0\).
They can be assumed to have vertices \((0,0)\), \((0,y_1)\) and \((0,w_2)\) for integers \(y_1 \in [0,w_2/2]\).

\begin{corollary}\label{cor:hilbert}
The number of (degenerate and non-degenerate) triangles, up to affine equivalence, which are a subset of \([0,n]^2\) is equal to \(|nQ \cap \ZZ^4|\) where \(Q\) is the four-dimensional rational simplex
\[
Q \coloneqq \conv\left(\left(\frac{1}{2},0,0,0\right),\left(0,\frac{1}{2},0,0\right),\left(0,0,\frac{1}{2},0\right),(0,0,0,1),(-1,-1,-1,-1)\right).
\]
Furthermore, the generating function of this sequence is the Hilbert series of a degree 8 hypersurface in \(\PP(1,1,1,2,2,2)\).
\end{corollary}
\begin{proof}
First we observe that when \(w_1>0\)
\begin{align*}
    \sum_{w_2=w_1}^\infty s^{w_2}|\cT_{w_1,w_2}|
    = & s^{w_1}\left[\frac{w_1^2(1+s)}{4(1-s)} + \frac{w_1}{2} + \frac{5+6s+5s^2}{8(1-s^2)}+(-1)^{w_1}\frac{3+2s+3s^2}{8(1-s^2)} \right]
\end{align*}
and when \(w_1=0\)
\begin{align*}
    \sum_{w_2=w_1}^\infty s^{w_2}|\cT_{w_1,w_2}|
    = & \frac{1}{(1-s^2)(1-s)}.
\end{align*}
From this we can compute that the generating function of \(|\cT_{w_1,w_2}|\) is
\begin{align}\label{eq:rectangle_gen_func}
    \sum_{w_1=0}^\infty \sum_{w_2=w_1}^\infty t^{w_1}s^{w_2}|\cT_{w_1,w_2}|
    = & \frac{-s^7t^4 + s^6t^3 - s^5t^2 + s^4t^3 - s^3t + s^2t^2 - st + 1}{(1-s)^2(1+s)(1-st)^3(1+st)}.
\end{align}
Let \(T\) be a triangle with first and second width \(w_1\) and \(w_2\) respectively.
To be equivalent to a subset of a square \(T\) must have the same width in two linearly independent directions.
Therefore, the smallest square which \(T\) is equivalent to a subset of has side length \(w_2\).
This means that the generating function of the number of triangles which can be contained in \([0,w_2]^2\) and no smaller square is obtained by setting \(t=1\) in \eqref{eq:rectangle_gen_func}.
This gives
\[
\frac{1-s^8}{(1-s^2)^3(1-s)^2}.
\]
To compute instead the generating function of the number of triangles which are affine equivalent to a subset of \([0,w_2]^2\) (and possibly smaller squares) we divide this by \((1-s)\).
The result is the Hilbert series of a degree 8 hypersurface in the weighted projective space \(\PP(1,1,1,2,2,2)\) and also the Ehrhart series of \(Q\), that is
\[
\Ehr_Q(t) \coloneqq \sum_{n=0}^{\infty} |nQ \cap \ZZ^4|s^n = \frac{1-s^8}{(1-s^2)^3(1-s)^3}.
\]
The Ehrhart series of \(Q\) can be found using computational algebra.
By definition of the Ehrhart series this proves the result.
These two descriptions of the generating function are related via mirror symmetry in the sense of \cite{mirror_symmetry}.
Let \(f\) be a Laurent polynomial which is a mirror partner for a degree 8 hypersurface in \(\PP(1,1,1,2,2,2)\), then \(Q\) is the dual of the Newton polytope of \(f\).
\end{proof}

Finally, we note some sub-sequences which appear in Sloane's On-Line Encyclopedia of Integer Sequences (OEIS) \cite{OEIS}.
We refer to these sequences by their OEIS reference number.
The sequence counting triangles which have an edge of lattice length \(n\) and are a subset of \([0,n]^2\) is A140144.
The sequence counting these same triangles but excluding those with zero volume is A135276.
Let \((a_n)_{n \geq 0}\) be the sequence counting triangles which are a subset of an \(n\times n\) square and no smaller and which also have no edge of lattice length \(n\).
Then the sequence \((a_n-a_{n-1})_{n\geq 1}\) is the rounded up staircase diagonal on the natural numbers, shown in Figure~\ref{fig:staircase}, which is A080827. 
We distinguish between triangles with an edge of lattice length \(n\) and not since these are exactly the triangles of type (A) and those with the longest possible edge contained in a square.
It is not obvious why any of these sequences should coincide.

\begin{figure}[ht]
\begin{tikzpicture}[x=0.8cm,y=0.8cm]
\tikzstyle{every node}=[font=\small]
\draw (0,0) -- (1,0) -- (1,1)--(2,1)--(2,2)--(3,2)--(3,3)--(4,3)--(4,4)--(5,4)--(5,5);
\filldraw[fill=white] (0,0) circle (0.3);
\filldraw[fill=white] (1,0) circle (0.3);
\filldraw[fill=white] (1,1) circle (0.3);
\filldraw[fill=white] (2,1) circle (0.3);
\filldraw[fill=white] (2,2) circle (0.3);
\filldraw[fill=white] (3,2) circle (0.3);
\filldraw[fill=white] (3,3) circle (0.3);
\filldraw[fill=white] (4,3) circle (0.3);
\filldraw[fill=white] (4,4) circle (0.3);
\filldraw[fill=white] (5,4) circle (0.3);
\filldraw[fill=white] (5,5) circle (0.3);
\def\numberlist{{1,3,6,10,15,21,2,5,9,14,20,27,4,8,13,19,26,34,7,12,18,25,33,42,11,17,24,32,41,51,16,23,31,40,50,61}} 
  \foreach \x in {0,...,5}
    \foreach \y in {0,...,5}{ 
	\pgfmathtruncatemacro{\tmp}{\x + \y*6}
       \pgfmathtruncatemacro{\label}{\numberlist[\tmp]}
       \node[]  (\x\y) at (\x,\y) {\label};
} 
\end{tikzpicture}
\centering
\caption{The rounded up staircase diagonal on the natural numbers.}
\label{fig:staircase}
\end{figure}

\section{Ehrhart theory of lattice triangles}
\label{sec:ehrhart}

In this section we discuss the set of pairs \((b,i)\) where \(b\) and \(i\) are the number of boundary and interior points of a non-degenerate lattice triangle.
For a polygon \(P\) let \(b(P)\) denote the number of boundary lattice points and \(i(P)\) the number of interior lattice points.
The Ehrhart polynomial of a lattice polygon is
\[
\ehr_P(n) = \left(i(P)+\frac{b(P)}{2}-1\right)n^2 + \frac{b(P)}{2}n + 1
\]
which counts the number of lattice points in \(nP\).
Therefore, studying pairs \((b(P),i(P))\) is equivalent to studying the Ehrhart theory of lattice polygons.
In \cite{HNO} it was proven that for integers \(c\geq 1\) the cones 
\[
\sigma_c^\circ \coloneqq \left\{(b,i) \in \RR^2_{\geq0} : \frac{c-1}{2}b-(c-1) < i < \frac{c}{2}b - c(c+2)\right\}
\] 
contain no points \((b(T),i(T))\) where \(T\) is a lattice triangle.
The boundaries of consecutive cones are parallel lines and describe strips of the plane where points \((b(T),i(T))\) can fall.
Hofscheier--Nill--\"Oberg observed that there are periodic patterns in the points in these strips which can be observed in Figure~\ref{fig:b_i_main_plot} and \ref{fig:b_i_coloured}.

In Figure~\ref{fig:b_i_main_plot} and we plot points \((b(T),i(T))\) and denote which are realised by a triangle with or without an edge of lattice length \(\width^2(T)\).
The periodic strips seem to be made up of triangles with a long edge,
suggesting that the patterns observed relate somehow to a triangle having one relatively long edge.
Additionally, Figure~\ref{fig:b_i_coloured} shows the same plot only now it is coloured by the smallest and largest first and second width of lattice triangles realising each point.
In all of these plots, clear patterns emerge in the strips away from the \(i\)-axis.
When coloured by first width, each strip has its own consistent colour.
When coloured by second width, each strip has an even gradient of colour.
This suggests that the points in a given strip are realised mainly by triangles with fixed \(w_1\) and evenly increasing \(w_2\).
We make all of this precise in the following result.
\begin{figure}
\centering
    \begin{tikzpicture}
    \begin{axis}[axis lines=left, xmin=0, ymin=-1, xlabel = {Number of boundary points}, ylabel = {Number of interior     points}, width=0.9\textwidth]
    \addplot[only marks, mark=*, mark size=0.7pt]
    table{long_edge_ehr.dat};
    \addplot[only marks, mark=x, mark size=1.8pt]
    table{short_edge_ehr.dat};
    \end{axis}
    \end{tikzpicture}
\caption{Number of boundary and interior points of lattice triangles with multi-width \((w_1,w_2)\) where \(0 < w_1 \leq w_2 \leq 100\) and where the number of interior or boundary points does not exceed 100.
Dots (resp. crosses) denote points which are realised by triangles with (resp. without) an edge of length \(w_2\).
Some points are realised by both.}
\label{fig:b_i_main_plot}
\end{figure}
\begin{figure}
\centering
\subfloat[Coloured by the smallest first width of triangles realising points.]{
    \begin{tikzpicture}
    \begin{axis}[axis lines=left, xmin=0, ymin=-1, xlabel = {}, ylabel = {}, width=0.48\textwidth]
    \addplot+[only marks, mark=*, mark size=0.5pt, scatter, point meta=explicit]
    table[meta=w1_min]{b_i_to_colour2.dat};
    \end{axis}
    \end{tikzpicture}
}
\subfloat[Coloured by the largest first width of triangles realising points]{
    \begin{tikzpicture}
    \begin{axis}[axis lines=left, xmin=0, ymin=-1, xlabel = {}, ylabel = {}, width=0.48\textwidth]
    \addplot+[only marks, mark=*, mark size=0.5pt, scatter, point meta=explicit]
    table[meta=w1_max]{b_i_to_colour2.dat};
    \end{axis}
    \end{tikzpicture}
}\\
\subfloat[Coloured by the smallest second width of triangles realising points.]{
    \begin{tikzpicture}
    \begin{axis}[axis lines=left, xmin=0, ymin=-1, xlabel = {}, ylabel = {}, width=0.48\textwidth]
    \addplot+[only marks, mark=*, mark size=0.5pt, scatter, point meta=explicit]
    table[meta=w2_min]{b_i_to_colour2.dat};
    \end{axis}
    \end{tikzpicture}
}
\subfloat[Coloured by the largest second width of triangles realising points.]{
    \begin{tikzpicture}
    \begin{axis}[axis lines=left, xmin=0, ymin=-1, xlabel = {}, ylabel = {}, width=0.48\textwidth]
    \addplot+[only marks, mark=*, mark size=0.5pt, scatter, point meta=explicit]
    table[meta=w2_max]{b_i_to_colour2.dat};
    \end{axis}
    \end{tikzpicture}
}
\caption{The plot shown in Figure~\ref{fig:b_i_main_plot}, now coloured according to the widths of triangles which realise each point.}
\label{fig:b_i_coloured}
\end{figure}
\begin{proposition}
Let \(S_w\) be the set of points \((b(T),i(T))\) where \(T\) has width \(w\) with respect to the normal to one of its edges.
Then we have
\[
S_w \subseteq \left\{(b,i) \in \ZZ^2: 1-w^2 \leq i-\frac{w-1}{2}b \leq 1-w\right\}
\]
and 
\[
S_w +\left(w,\frac{w(w-1)}{2}\right) \subseteq S_w 
\]
In fact, there are finitely many triangles \(T_1, \dots, T_r\), which have width \(w\) with respect to a normal to an edge, such that all points of \(S_w\) can be written \((b(T_j),i(T_j)) + k(w,w(w-1)/2)\) for some positive integer \(k\).
In other words, the points of \(S_w\) form a periodic pattern in a strip of the plane, generated by a finite collection of triangles.
The set \(S\) of points \((b(T),i(T))\) where \(T\) is a lattice triangle is the union of these \(S_w\) for \(w \geq 1\).
\end{proposition}

\begin{proof}
It is immediate that \(S\) is the union of the sets \(S_w\) for \(w \geq 1\) so it remains to prove the properties of \(S_w\).
Let \(T\) be a lattice triangle and say that \(u\) is a normal to one of its edges such that \(\width_u(T)=w\).
There is an affine transformation which takes one of the end points of this edge to the origin, the other to the positive \(x\)-axis and the third vertex above the \(x\)-axis.
We now have a triangle of the form \(\conv((0,0),(l,0),(a,w))\) for some positive integer \(l\).

Pick's Theorem tells us that the normalised volume of a lattice polygon \(P\) is \(2i(P)+b(P)-2\).
The normalised volume of \(T\) is \(lw\) and its number of boundary points is at most \(l+2w\) and at least \(l+2\).
Notice that \(i(T)-(w-1)b(T)/2\) is equal to \((2i(T)+b(T))/2-wb(T)/2\) and, using Pick's Theorem, this is equal to \((lw+2)/2 - wb(T)/2\).
Therefore, the bounds on \(b(T)\) give
\[
1-w^2 \leq i(T)-\frac{w-1}{2}b(T) \leq 1-w.
\]

Now consider the triangle \(T' = \conv((0,0),(l+w,0),(a,w)\).
Its volume is \((l+w)w\) and it has \(b(T) + w\) boundary points.
From Pick's Theorem we can compute that it has \(i(T) + w(w-1)/2\) interior points.
This shows that \(S_w + (w,w(w-1)/2)\) is a subset of \(S_w\).

To show that there is a finite collection of triangles generating \(S_w\), consider the triangle \(\conv((0,0),(l-kw,0),(a,w))\), where \(k\) is the integer such that \(l-kw\) is the smallest positive integer possible.
This has volume \((l-kw)w\), \(b(T)-kw\) boundary points and \(i(T)-kw(w-1)/2\) interior points.
Therefore, if \(k>0\) the point \((b(T),i(T))\) can be obtained from another point in \(S_w\) by adding \((w,w(w-1)/2)\).
If \(k=0\) then \(l \in (0,w]\).
We may assume by a shear that \(a \in [0,w)\).
Therefore, there are finitely many choices for \(T\).
\end{proof}

Notice that the lower and upper boundaries of \(S_w\) are in the same hyperplane as the upper boundary of \(\sigma_{w-1}\) and lower boundary of \(\sigma_w\) respectively.
However, this result does not reproduce the result of \cite{HNO} since, as subsets of the real plane, the strips containing the sets \(S_w\) and the cones \(\sigma_c\) intersect non-trivially.
This intersection could potentially contain points \((b,i)\) if it were not for the proof that the cones are empty.
\bibliographystyle{alpha}
\bibliography{TrianglesBib}{}
\end{document}